\documentclass[conference]{IEEEtran}
\IEEEoverridecommandlockouts
\usepackage{cite}
\usepackage{amsmath,amssymb,amsfonts, amsthm}
\usepackage{graphicx}
\usepackage{textcomp}
\usepackage{xcolor}
\newtheorem{theorem}{Theorem}
\theoremstyle{definition}
\def\BibTeX{{\rm B\kern-.05em{\sc i\kern-.025em b}\kern-.08em
    T\kern-.1667em\lower.7ex\hbox{E}\kern-.125emX}}
\begin{document}

\title{Marrying Compressed Sensing and Deep Signal Separation\\
%{\footnotesize \textsuperscript{*}Note: Sub-titles are not captured in Xplore and
%should not be used}
\thanks{Work performed in Southwest Research Institute (SwRI)}
}

\author{\IEEEauthorblockN{Truman Hickok}
\IEEEauthorblockA{\textit{Southwest Research Institute (SwRI)} \\
San Antonio, Texas, USA \\
truman.hickok@swri.org}
\and
\IEEEauthorblockN{Sriram Nagaraj}
\IEEEauthorblockA{\textit{Southwest Research Institute (SwRI)} \\
San Antonio, Texas, USA \\
sriram.nagaraj@swri.org}

}

\maketitle

\begin{abstract}
Blind signal separation (BSS) is an important and challenging signal processing task. Given an observed signal which is a superposition of a collection of unknown (hidden/latent) signals, BSS aims at recovering the separate, underlying signals from only the observed mixed signal. As an underdetermined problem, BSS is notoriously difficult to solve in general, and modern deep learning has provided engineers with an effective set of tools to solve this problem. For example, autoencoders learn a low-dimensional hidden encoding of the input data which can then be used to perform signal separation. In real-time systems, a common bottleneck is the transmission of data (communications) to a central command in order to await decisions. Bandwidth limits dictate the frequency and resolution of the data being transmitted. To overcome this, compressed sensing (CS) technology allows for the direct acquisition of compressed data with a near optimal reconstruction guarantee. This paper addresses the question: can compressive acquisition be combined with deep learning for BSS to provide a complete acquire-separate-predict pipeline? In other words, the aim is to perform BSS on a compressively acquired signal directly without ever having to decompress the signal. We consider image data (MNIST and E-MNIST) and show how our compressive autoencoder approach solves the problem of compressive BSS. We also provide some theoretical insights into the problem.
\end{abstract}


\begin{figure*}[!htbp]
\centering
\includegraphics[width=\textwidth, height=3cm, keepaspectratio]{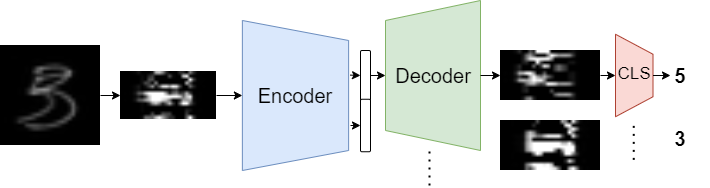}
\caption{Overview of our method. The encoder transforms the compressed multisource image into a set of compact vectors. The decoder then independently transforms each vector into a compressed, single-source image. Finally, the classifier predicts the identity of the compressed, single-source image produced by the frozen model once BSS training is complete.}
\label{fig:overview}
\end{figure*}

\section{Introduction}
Blind signal separation (BSS) is an important and challenging signal processing task. Briefly, given an observed signal which is a superposition (more generally, a nonlinear mixture) of a collection of unknown (hidden/latent) signals, BSS aims to recover the separate, underlying signals from only the observed mixed signal. As an underdetermined problem, BSS is notoriously difficult to solve in general, and modern deep learning has provided engineers with an effective set of tools to solve this problem.

Real-time BSS is a sophisticated signal processing technique that separates mixed signals into their individual sources without prior knowledge of the source signals or the mixing process. It is particularly effective in scenarios where multiple signals, such as audio or radio transmissions, are overlapping or intertwined.

Some applications of BSS in avionics include:
\begin{itemize}
\item BSS algorithms can enhance the analysis of cockpit voice and flight data recordings. By isolating and separating relevant audio and sensor data streams, BSS improves post-incident analysis and investigation processes (\cite{b12}).
\item BSS algorithms can enhance communication signal processing in avionics systems. By separating mixed communication signals, such as those from different antennas or frequencies, BSS improves signal clarity, reliability, and spectral efficiency (\cite{b13}).
\item  Real-time BSS can handle multi-modal signals simultaneously, providing a comprehensive understanding of the operational environment, which would be especially important in reconnaissance missions that involve analyzing various signal modalities, such as audio, video, and radar.
\end{itemize}

Modern deep machine learning (deep ML) effectively addresses the BSS problem, and a variety of modeling methods have been proposed in this context . For example, the Autoencoder class of models (\cite{b3}, including variational autoencoders (aka VAE models) learn a low-dimensional hidden encoding of the input data which can then be used to perform denoising, signal separation and novel signal generation. Generative adversarial networks (GANs) consist of two competing networks, one for generating novel samples and the other to label the novel signal. Upon training, a GAN can be used to separate multiple components within the signal and/or for denoising. 

In real-time systems, a common bottleneck is the transmission of data (communications) to a central command in order to await decisions, many of which may be routine. Bandwidth limits also dictate the frequency and resolution of the data being transmitted, and the communication bottleneck from the active device (the ``edge”) to the consumer of the information needs to be surmounted. In this regard, compressed sensing (CS) technology \cite{b2,b10} allows for the direct acquisition of compressed data with a near optimal reconstruction guarantee. By acquiring already compressed data samples, there is no need for an acquire-process-compress-transmit approach, but rather, the approach becomes acquire/compress simultaneously-process-transmit. The compression is linear (as opposed to thresholding, which is usually nonlinear), and the signals can be decompressed with minimal loss of fidelity. Using compressed sensing methods can therefore help mitigate the communications limitations between edge device and data/signal recipients. However, an open question is: can compressive acquisition be combined with deep learning to provide a complete acquire-separate-predict pipeline? Answering this question is the main aim of this paper.
\subsection{Compression and Autoencoding}
One can view signal compression/restoration from the point of view of autoencoding. Indeed, classical signal compression scheme based on transforms such as DCT/wavelets etc. are a means of encoding/decoding the relevant signals in the corresponding bases. More generally, linear adaptive thresholding methods for compression based on singular value decomposition (SVD) can also be viewed as encoding/decoding schemes. Symmetric autoencoders generalize this framework by using deep learning to learn the weights and biases in the hidden representation of signals (i.e. the encoding), and the decoder is simply defined by the transpose of the weights/biases of the encoder. From this perspective, autoencoders are powerful nonlinear, adaptive compression schemes.
\subsection{Prior Work}
Compressive sensing by itself is a mature field (\cite{b2,b10}), and we build upon this body of work. Compressive sensing stands in contrast to traditional methods that require extensive sampling of signals at or above the Nyquist rate.  Indeed, Compressive sensing exploits the sparsity or compressibility of signals to accurately reconstruct them from far fewer samples than traditionally believed possible. Instead of acquiring signals at their full resolution, CS employs random or structured measurements, followed by advanced reconstruction algorithms to recover the original signal. This paradigm shift has profound implications for various fields, including telecommunications, imaging, and data compression.

In communication systems, compressive sensing enables the efficient acquisition of signals using fewer samples. By employing CS techniques, such as random projections or structured sensing matrices, systems can capture the essential information of a signal with reduced sampling requirements. Compressive sensing allows for low-rate sampling and subsequent reconstruction of signals, enabling the transmission of data over bandwidth-limited channels. This is particularly advantageous in scenarios where available bandwidth is constrained, such as wireless communication networks or satellite links.

Nonlinear extensions of the compressive sensing approach to manifold valued data have also been studied (\cite{b7,b8,b9}). Compressive learning (\cite{b1,b4}) attempts to perform machine learning tasks in the compressed domain. Signal separation using deep learning methods has also been done (\cite{b5,b6}). These methods rely on (variational) autoencoder/GAN class of models to learn latent encodings of the mixed signals. In \cite{b11}, a deep Bayesian approach is proposed for signal separation using a Gaussian mixture model. However, the problem addressed in this paper is novel in that we are interested in separation of compressed, mixed signals.
\section{Signal Model}
We describe briefly the signal model and mathematical assumptions of this work. We then state and prove some fundamental theoretical results.

\subsection{Compressive Sensing}
Given a signal $x\in \mathbb{R}^D,$ we say it is $k$-sparse if there exists a basis represented by the square matrix $\Psi$ such that $\Psi x$ has at most $k$ nonzero components. Given a $k$ sparse vector $x$, CS theory dictates that a $d=O(k \text{ log}(\frac{D}{k}))$ nonadaptive samples of $x$ are sufficient to reconstruct $x$, and furthermore, that a random matrix $\Phi$ satisfying the so-called restricted isometry property (RIP) can be used for this purpose. In terms of reconstruction, a variety of algorithms have been designed to achieve good reconstruction properties (COSAMP etc.). Given a generic reconstruction algorithm, we write it as a map $R: \mathbb{R}^d \rightarrow \mathbb{R}^D$. In our analysis, the particular for or complexity of the reconstruction algorithm is not of primary importance. Indeed, our analysis is focused more on the combined effect of reconstruction and signal separation, and the bounds we derive would be valid for any generic reconstruction method.

\subsection{Signal Separation}
Assume that we have a collection $\mathcal{X}=\{x_1,\ldots,x_N\}$ of $k$-sparse vectors and a mixing operator $\circledast$ that is a closed, associative, commutative operator that distributes over addition in $\mathbb{R}^D$. In other words, $\circledast$ acts on vectors $x,y,z\in \mathbb{R}^D$ such that $x \circledast y = y \circledast x$, $x\circledast (y\circledast z) = (x\circledast y)\circledast z$ and $ x\circledast (y+z) = x\circledast y + x\circledast z$. The mixing operator corresponds to the phenomenon of signal mixing, and the objective of signal separation is to estimate the components given the mixture, i.e., given $x^{\circledast} = x_{i_1} \circledast x_{i_2} \ldots \circledast x_{i_k}$ with $1 \leq i_1,\ldots, i_k \leq N,$ return an estimate $\begin{pmatrix}\hat{x}_{i_1}\\\hat{x}_{i_2}\\ \ldots \\ \hat{x}_{i_k}\end{pmatrix}$ of $x = \begin{pmatrix}x_{i_1}\\ x_{i_2}\\ \ldots \\ x_{i_k}\end{pmatrix}$. Thus, given a known number ($k$) of mixing components, a separation operator is a (possibly nonlinear) map $S: \mathbb{R}^D \rightarrow \mathbb{R}^{kD}$ such that $\|x - Sx^{\circledast}\|$ minimized. Given the set $\mathcal{X},$ we denote by $\bar{\mathcal{X}_k} = \{x_{i_1} \circledast \ldots \circledast x_{i_k}: 1\leq i_1,\ldots i_k \leq N\},$ i.e., the set of $k-$fold mixed signals drawn from signals in $\mathcal{X}.$ Note that $\bar{\mathcal{X}_1} = \mathcal{X}.$
\subsection{Compressive Signal Separation}
Given a compressive sensing matrix $\Phi$, we make the following assumption: \begin{equation}\label{eq:eq1}\Phi(x_{i_1} \circledast x_{i_2} \ldots \circledast x_{i_k}) = \Phi x_{i_1} \circledast \Phi x_{i_2} \ldots \circledast \Phi x_{i_k},\end{equation} for all $1 \leq i_1,\ldots, i_k \leq N$, where $\circledast$ is also defined in $\mathbb{R}^d$ on the righthand side of Equation \ref{eq:eq1}. Given the set of uncompressed signals $\mathcal{X},$ we denote by $\Phi\bar{\mathcal{X}_k} = \{\Phi x_{i_1} \circledast \ldots \circledast \Phi x_{i_k}: 1\leq i_1,\ldots i_k \leq N\}.$ A reconstruction operator is \emph{perfect} if $R\Phi x = x$ for all $x\in \bigcup_{i=1}^{N} \bar{\mathcal{X}}_i.$ An \emph{Oracle} $\mathcal{O}$ is a separation operator which achieves perfect separation, i.e., $\mathcal{O}(x^{\circledast}) = x.$ Given access to an oracle $\mathcal{O}$, we can state the following theorem, which says that an oracle in the compressed space can be obtained from an oracle in the uncompressed space and a perfect reconstruction operator.
\begin{theorem}\label{thm:th1}Given a known number of mixing components, $k$, the $k\times k$ identity matrix $I_k$, an oracle $\mathcal{O}$ and a perfect separation operator $R$, we have that the operator $\hat{\mathcal{O}} = (\Phi\otimes I_k)\mathcal{O} R$ is an oracle on $\Phi\bar{\mathcal{X}_k},$ conversely, every oracle on $\Phi\bar{\mathcal{X}_k}$ is of this form.
\end{theorem}
\begin{proof}
Given $\Phi x^\circledast \in  \Phi\bar{\mathcal{X}_k}, \Phi x^\circledast = \Phi x_{i_1} \circledast \ldots \circledast \Phi x_{i_k},$ we have that \begin{equation}R\Phi x^\circledast = x^\circledast = x_{i_1} \circledast \ldots \circledast x_{i_k},\end{equation} since $R$ is perfect, and \begin{equation}\mathcal{O}x^\circledast = x= \begin{pmatrix}x_{i_1}\\ x_{i_2}\\ \ldots \\ x_{i_k}\end{pmatrix}.\end{equation} Now \begin{equation}(\Phi\otimes I_k)\mathcal{O} R \Phi x^\circledast = \begin{pmatrix}\Phi, O, \ldots, O\\O, \Phi, \ldots, O\\ \ldots \\ O, \ldots, O, \Phi\end{pmatrix}\mathcal{O} x^\circledast = \begin{pmatrix}\Phi x_{i_1}\\ \Phi x_{i_2}\\ \ldots \\ \Phi x_{i_k}\end{pmatrix},\end{equation} which is the vector of separated, compressed signal components. Conversely, if $\hat{\mathcal{O}}$ is an oracle on $\Phi\bar{\mathcal{X}_k},$ then \begin{equation}\hat{\mathcal{O}}\Phi x^\circledast = \begin{pmatrix}\Phi x_{i_1}\\ \Phi x_{i_2}\\ \ldots \\ \Phi x_{i_k}\end{pmatrix} = (\Phi\otimes I_k)\begin{pmatrix}x_{i_1}\\ x_{i_2}\\ \ldots \\ x_{i_k}\end{pmatrix}.\end{equation} If $\mathcal{O}$ is an oracle on $\bar{\mathcal{X}_k},$ then $\begin{pmatrix}x_{i_1}\\ x_{i_2}\\ \ldots \\ x_{i_k}\end{pmatrix} = \mathcal{O}x^\circledast,$ so that \begin{equation}\hat{\mathcal{O}}\Phi = (\Phi\otimes I_k)\mathcal{O}\Rightarrow \hat{\mathcal{O}} = (\Phi\otimes I_k)\mathcal{O}\Phi^{\dagger} = (\Phi\otimes I_k)\mathcal{O}R\end{equation}where $\Phi^{\dagger}$ is the pseudoinverse of $\Phi,$ which is $R$ since $R$ is perfect. 
\end{proof}
The next result is applicable for non-perfect reconstruction with an oracle separation operator.
\begin{theorem}Given a known number of mixing components, $k$, an oracle $\mathcal{O}$ and a reconstruction operator $R$ such that $\|x-R\Phi x\|<\epsilon$ for all $x\in \mathcal{X}$, the operator $\hat{\mathcal{O}} = (\Phi\otimes I_k)\mathcal{O} R$ satisfies $\|(R\otimes I_k)\hat{\mathcal{O}}\Phi x^{\circledast} - x\|< L\epsilon,$ where $L$ is a constant depending on $k$.
\end{theorem}
\begin{proof}
From the calculations in the proof of Theorem \ref{thm:th1}, we have that \begin{equation}\hat{\mathcal{O}}\Phi x^{\circledast}=\begin{pmatrix}\Phi x_{i_1}\\ \Phi x_{i_2}\\ \ldots \\ \Phi x_{i_k}\end{pmatrix}.\end{equation}Thus\begin{equation*}(R\otimes I_k)\hat{\mathcal{O}}\Phi x^{\circledast} = \begin{pmatrix}R, O, \ldots, O\\O, R, \ldots, O\\ \ldots \\ O, \ldots, O, R\end{pmatrix}\begin{pmatrix}\Phi x_{i_1}\\ \Phi x_{i_2}\\ \ldots \\ \Phi x_{i_k}\end{pmatrix}= \begin{pmatrix}R\Phi x_{i_1}\\ R\Phi x_{i_2}\\ \ldots \\ R\Phi x_{i_k}\end{pmatrix}, \end{equation*}hence\begin{equation}(R\otimes I_k)\hat{\mathcal{O}}\Phi x^{\circledast} - x = \begin{pmatrix}R\Phi x_{i_1}-x_{i_1}\\ R\Phi x_{i_2}-x_{i_2}\\ \ldots \\ R\Phi x_{i_k}-x_{i_k}\end{pmatrix},\end{equation}so that $\|(R\otimes I_k)\hat{\mathcal{O}}\Phi x^{\circledast} - x\|< \sqrt{k}\epsilon,$ where the norm $\|\cdot\|$ is the usual Euclidean norm. Using a different norm gives rise to a different constant in place of $\sqrt{k}$, though the scaling with $\epsilon$ is unchanged.
\end{proof}
These two results indicate that the error in a non-oracle compressive signal separation scenario would be dominated by the error in the signal separation operator (a usually nonlinear map), while the error due to compressive reconstruction is bounded by a linear factor in the reconstruction error. 
\section{Methods}

\subsection {Problem Setup}

Let $\Phi x=\Phi x_1\circledast \Phi x_2$ represent the input mixed, compressed image. The objective is to estimate the compressive separation operator $S$ such that:
\[ S\Phi x = (\widehat{\Phi x_1}, \widehat{\Phi x_2}) \]
where $\widehat{\Phi x_1}, \widehat{\Phi x_2}$ are the desired outputs, each a separate, compressed estimate of the representation of the original image components.

\subsection {Deep Learning Architecture}

To tackle this problem, we opted for a simplified approach in the design of the deep learning architecture, aiming to maintain a balance between model complexity and computational efficiency. The chosen architecture is a basic autoencoder, a type of neural network known for its ability in data compression and reconstruction tasks. The autoencoder is designed to process the compressed, mix image, extract its features, and separate the mixed content into distinct latent representations corresponding to each component of the image. The architecture can be divided into two main stages: encoding and decoding. 

1. Encoding: The encoder part of the autoencoder takes the mixed, compressed image \(\Phi x\) as input and applies a series of transformations to produce two sets of latent vectors, \(\Phi z_1\) and \(\Phi z_2\). These vectors encapsulate the essential information needed to reconstruct the individual components of \(\Phi x\). The encoding process can be represented as:
   \[ \Phi z_1, \Phi z_2 = \mathcal{E}(\Phi x) \]

2. Decoding: Following encoding, each latent vector is separately decoded to recover the compressed images. The decoding process for each component can be represented as:
   \[ \widehat{\Phi x_1} = \mathcal{D}_1(\Phi z_1) \]
   \[ \widehat{\Phi x_2} = \mathcal{D}_2(\Phi z_2) \]

This architecture leverages the autoencoder’s ability to learn efficient representations and allows for the separation and reconstruction of the mixed image components in a compressed form. The simplicity of this approach facilitates a more straightforward training process and reduces the computational resources required, making it an efficient solution for the problem at hand. An illustration of our method can be found in Figure~\ref{fig:overview}.

\section{Implementation Details}

Per existing standards, we used a symmetric autoencoder and tuned the learning rate, the size of the latent space, the number of layers in the neural network, and the size of each layer. We also used a cosine learning rate scheduler to stabilize training. Models were trained for up to 3000 epochs at a batch size of 128 mixed images, with the best-performing model being saved for evaluation at test-time. 

After the completion of the training phase, we froze the model and trained a classifier on the compressed, separated outputs to evaluate the effectiveness of the separation process. We report the top-1 accuracy on a held-out test set.

We use random Bernoulli matrices to perform the compressive sensing, as these matrices are known to obey the RIP.

\begin{table}[!htp]
\centering
\caption{MNIST results}
\label{tab:mnisttable}
\begin{tabular}{l c r}
\hline
Model & \# parameters & Accuracy \\
\hline
No compression & 1 mil & 94.7\% \\
25\% sensing rate & 1 mil & 84.7\% \\
25\% sensing rate & 2 mil & 85.7\% \\
50\% sensing rate & 1 mil & 92.8\% \\
50\% sensing rate & 2 mil & 94.5\% \\
\hline
\end{tabular}
\end{table}

\begin{table}[!htp]
\centering
\caption{E-MNIST results}
\label{tab:emnisttable}
\begin{tabular}{l c r}
\hline
Model & \# parameters & Accuracy \\
\hline
No compression & 2 mil & 82.5\% \\
25\% sensing rate & 2 mil & 70.4\% \\
25\% sensing rate & 4 mil & 71.5\% \\
50\% sensing rate & 2 mil & 79.9\% \\
50\% sensing rate & 4 mil & 82.0\% \\
\hline
\end{tabular}
\end{table}

The initial phase of our experiments leveraged the MNIST dataset, which comprises 70,000 images of handwritten digits (0-9), including 10,000 images designated for testing. Subsequently, we extended our experiments to the Extended MNIST (E-MNIST) dataset. This variant encompasses over 800,000 images of handwritten letters and digits spanning 47 unique classes and includes over 100,000 images for testing.

\section{Results}

To evaluate our compressed BSS models, we trained an autoencoder model for each dataset which takes a mixed, non-compressed image as input and outputs two separated, non-compressed images. For fair comparison, these models were designed to have the same parameter count (which, however, means less neurons since the inputs are much larger) and hyperparameter sweep as the compressed BSS models. We report classification accuracies according to each model's parameter count and the compressed sensing rate of each model's inputs and outputs (a 25\% sensing rate indicates that each mixed and separated image will have a height and width of 7x28 instead of 28x28, for example).

\subsection{MNIST Results}

MNIST results are shown in Table~\ref{tab:mnisttable}. The non-compressed autoencoder model, which operates on the full pixel set of the MNIST images, achieved a remarkable top-1 accuracy of approximately 95\%. This result underscores the model's capability to effectively separate mixed signals into their constituent components when provided with the entirety of the input data. The high level of accuracy serves as a testament to the potential of deep learning architectures in addressing the BSS problem, setting a high standard for our compressed sensing models.

In an attempt to match the performance of the non-compressed model while operating under bandwidth constraints, we introduced a compressed model with a 50\% sensing rate. This model was designed to process images that had been reduced to 50\% of their original pixel count, effectively halving the data volume without preliminary decompression. With an adequate increase in the model's parameter count to account for the reduced data availability, this compressed model succeeded in matching the benchmark top-1 accuracy of approximately 95\%. This parity in performance underscores the efficacy of combining compressed sensing with deep learning for BSS. It highlights the potential of our approach to maintain high accuracy in signal separation tasks even when data is significantly compressed, addressing a critical challenge in real-time signal processing applications.

We then explored the capabilities of a model operating at a 25\% sensing rate. While such aggressive compression posed a substantial challenge, the model demonstrated promising results. Although it could not fully match the top-1 accuracy of the non-compressed baseline, the performance of the 25\% sensing rate model nonetheless indicates a significant achievement. This finding suggests that, with further refinement and optimization, models operating at such low sensing rates could approach, if not reach, the performance levels of their less compressed counterparts.

\subsection{E-MNIST Results}

Following the analysis of our models on the MNIST dataset, we extended our investigation to the Extended MNIST (E-MNIST) dataset. The E-MNIST dataset presents a more formidable challenge, comprising a broader array of handwritten letters and digits, thereby testing the robustness and adaptability of our approach under more complex signal separation tasks. E-MNIST results are shown in Table~\ref{tab:emnisttable}.

The key result is that, despite the challenge presented by E-MNIST, our largest 50\% compressed model again rivals the performance of the non-compressed model, which indicates the scalability and potential of our proposed framework. It should be noted, however, that the difference in the number of parameters for the non-compressed model and the largest compressed model is more pronounced here than it was for MNIST, as the larger variant of the 50\% compressed model saw a 2 million-parameter increase instead of the 1 million-parameter increase for the 50\% compressed MNIST model.

Another key result here is that the performance gap between our largest 50\% compressed model and our largest 25\% compressed model remains relatively constant, which further indicates the resilience of our approach under varying degrees of data compression and task complexity.

\section{Summary and Conclusions}
The results of this work strongly suggest that compressive signal separation can be performed using deep learning methods without resorting to the brute force approach of complete reconstruction and then signal separation of the reconstructed signal. Indeed, on MNIST data with two blind sources, the compressive separation approach yielded results nearly as good as the uncompressed signal separation. At the same time, as indicated by our theoretical results, it is apparent that the brute force approach, in the presence of perfect reconstruction, yields a bound in terms of the overall fidelity of the compressive separation. In other words, in a perfect world, the brute force approach may be optimal. Therefore, the main bottleneck in compressive signal separation is not the compression aspect, but rather the signal separation aspect. 

On real data with colors (ImageNet, real data etc.), further work needs to be done in constructing more efficient signal separation models in the uncompressed domain, and then transfer those insights to compressive signal separation. The path forward in that direction appears to be with more sophisticated Bayesian models, and/or better tuned GANs. These models, however, come with significant computational cost for training on real data.

%\section*{Acknowledgment}

\vspace{12pt}
\end{document}